\documentclass{AIMS}
\usepackage{amsmath}
  \usepackage{paralist}
  \usepackage{graphics} 
  \usepackage{epsfig} 
\usepackage{graphicx}  \usepackage{epstopdf}
 \usepackage[colorlinks=true]{hyperref}
\hypersetup{urlcolor=blue, citecolor=red}
\usepackage{stackrel}

\def\currentvolume{X}
 \def\currentissue{X}
  \def\currentyear{200X}
   \def\currentmonth{XX}
    \def\ppages{X--XX}
     \def\DOI{10.3934/xx.xx.xx.xx}

\newtheorem{theorem}{Theorem}[section]
\newtheorem{corollary}[theorem]{Corollary}

\newtheorem{proposition}[theorem]{Proposition}

\theoremstyle{definition}
\newtheorem{definition}[theorem]{Definition}

\usepackage{todonotes}

\newcommand{\Qo}{Q(u_0)}
\newcommand{\RR}{\mathbb{R}}
\newcommand{\R}{\mathbb{R}}
\newcommand{\GL}{\mathrm{GL}}
\newcommand{\SO}{\mathrm{SO}}
\newcommand{\OO}{\mathrm{O}}
\newcommand{\Log}{\mathrm{Log}}
\newcommand{\Qu}{Q(u)}
\newcommand{\piQu}{\pi_{Q(u)}}
\newcommand{\Quo}{Q(u_0)}
\newcommand{\piQuo}{\pi_{Q(u_0)}}
\newcommand{\dQuo}{d_{\Quo}}

\newcommand{\verbose}[1]{}

\title[Anisotropic Covariance and Sub-Riemannian Geometry]
{Modelling Anisotropic Covariance using Stochastic Development and Sub-Riemannian
Frame Bundle Geometry}

\author[Stefan Sommer and Anne Marie Svane]{}

\subjclass{Primary: 53C17, 60E05; Secondary: 62H11.}
\keywords{differentiable geometry, statistics, sub-Riemannian geometry, stochastic processes, stochastic development, non-linear data analysis}

 \email{sommer@di.ku.dk}
 \email{amsvane@math.au.dk}

\begin{document}
\maketitle

\centerline{\scshape Stefan Sommer}
\medskip
{\footnotesize
  \centerline{University of Copenhagen}
 \centerline{Universitetsparken 5}
 \centerline{DK-2100 Copenhagen E}
} 

\medskip

\centerline{\scshape Anne Marie Svane}
\medskip
{\footnotesize
 \centerline{Aarhus University}
   \centerline{Ny Munkegade 118}
   \centerline{8000 Aarhus C, Denmark}
}
\bigskip

 \centerline{(Communicated by the associate editor name)}

\begin{abstract}
 We discuss the geometric foundation behind the use of stochastic processes in the frame bundle of a smooth manifold to build stochastic models with applications in statistical analysis of non-linear data. The transition densities for the projection to the manifold of Brownian motions developed in the frame bundle lead to a family of probability distributions on the manifold. We explain how data mean and covariance can be interpreted as points in the frame bundle or, more precisely, in the bundle of symmetric positive definite 2-tensors analogously to the parameters describing Euclidean normal distributions. 
 We discuss a factorization of the frame bundle projection map through this bundle, the natural sub-Riemannian structure of the frame bundle, the effect of holonomy, and the existence of subbundles where the H\"ormander condition is satisfied such that the Brownian motions have smooth transition densities. We identify the most probable paths for the underlying Euclidean Brownian motion and discuss small time asymptotics of the transition densities on the manifold. The geometric setup yields an intrinsic approach to the estimation of mean and covariance in non-linear spaces.
\end{abstract}

\section{Introduction}

Let $M$ be a smooth, connected, compact manifold of finite dimension. This paper is concerned with  stochastic modelling of data on $M$, a setting that arises in a range of applications, for example in medical imaging and computational anatomy \cite{grenander_computational_1998} where anatomical shapes reside in non-linear shape spaces.

In \cite{sommer_anisotropic_2015}, a class of probability distributions was introduced that generalizes the Euclidean normal distribution to the non-linear geometry on $M$. The aim is to allow statistical analysis of non-linear data by fitting the parameters of the distributions to data by a maximum likelihood approach. The method is intrinsic using development of stochastic processes in the frame bundle. In particular, it avoids the linearization of the manifold around a mean point that most non-linear statistical tools rely on.

In this paper, we will set forth the geometric foundation of the class of Brownian motions on $M$ and their transition distributions: we describe 
(1) the sub-Riemannian structure in the frame bundle of $M$ 
and the relation between the Fr\'echet mean on $M$ and the sub-Riemannian distance on $FM$; 
(2) the Eels-Elworthy-Malliavin construction of stochastic processes via the frame bundle; 
(3) a factorization through the bundle of symmetric positive definite matrices that represent infinitesimal covariance; 
(4) reduction to subbundles where the H\"ormander condition is satisfied and the density therefore smooth; 
(5) most probable paths for the generated stochastic processes; 
(6) small time asymptotics of the transition densities; and (7) applications in statistics.

\subsection{Statistical Background} \label{background}
In Euclidean space, it is natural to use the mean of a stochastic variable to describe the typical sample point, while the covariance captures the variation around this point.
However, the analogue for manifold valued stochastic variables is not clear. Many techniques use the Fr\'echet mean \cite{frechet_les_1948} as the natural definition of the mean point in $M$, see also the discussion \cite{pennec_intrinsic_2006}. Suppose $M$ is equipped with the Riemannian metric $g$, and let $d_g$ denote the associated geodesic distance function on $M$. Then the Fr\'echet mean of a stochastic variable $X$ on $M$ is a point $x_0 \in M$ that minimizes $\mathbb{E}d_g(x_0,X)^2$. If $\mathbb{E}d_g(x,X)^2$ is finite in one point $x\in M$, then such a minimizer exists, but it may not be unique. 
Given a set of data points $x_1,\dots, x_N$ on $M$, the Fr\'echet mean is usually estimated by a point $x_0\in M$ that realizes
\begin{equation}\label{squared}
  \textrm{argmin}_{x_0\in M}
 \sum_{i=1}^N d_g(x_0,x_i)^2
 \ .
\end{equation}

There are various attempts in the literature to define a notion of covariance or principal component analysis (PCA), see e.g. \cite{fletcher_statistics_2003,pennec_intrinsic_2006}. Typically, such approaches are based on a linearization of the manifold around the Fr\'echet mean, e.g. by using normal coordinates. However, if the covariance is not isotropic, it is arguably natural to replace  $d_g$ in \eqref{squared} by a distance measure that takes the anisotropy into account. Mimicking the Euclidean case, the covariance is usually represented by a symmetric positive definite matrix on $T_{x_0}M$ corresponding to an inner product. If the covariance is anisotropic, it is natural to measure distances in a metric that extends the inner product on $T_{x_0}M$ to all of $M$, rather than using $d_g$.
 
To enable this, we suggest  lifting the problem to the frame bundle. This requires that $M$ has a connection, typically the Levi-Civita connection associated with a Riemannian structure, and a fixed volume measure $\mu$. A point $u_0 $ in the frame bundle naturally defines an inner product $\sigma$ on $T_{x_0}M$, and given a connection in the frame bundle, there is a natural way of extending it to a (sub-Riemannian) metric on the whole frame bundle. This induces a (sub-Riemannian) distance function $d_{FM}$ on the frame bundle. We thus suggest replacing \eqref{squared} by
\begin{equation}\label{minformel}
\textrm{argmin}_{u_0=(x_0,\sigma)}  \sum_{i=1}^N \inf_{ \pi(u_i) = x_i} d_{FM} (u_0, u_i)^2
  -\frac{N}{2}\log(\mathrm{det}_g \sigma),
\end{equation}
where the infimum is taken over all points $u_i$ in the frame bundle projecting to $x_i$.  The last term assumes a Riemannian metric on $M$ and is inspired by the normalization factor in a Euclidean normal distribution. 
The advantage of this approach is that the minimizer encodes both the best descriptor $x_0\in M$ of the data and the covariance of the data represented by the inner product $\sigma$ on $T_{x_0}M$ that corresponds to the precision matrix, the inverse covariance. Informally, if $\sigma$ were kept fixed in \eqref{minformel}, a minimizing $x_0$ would correspond to the Fr\'echet mean with anisotropically weighted distance. However, while $\sigma$ can be parallel transported along curves in $M$, the curvature of the manifold generally prevents the notion of a globally fixed $\sigma$. We therefore need to minimize over  $x_0$ and $\sigma $ simultaneously in~\eqref{minformel}.

In \cite{sommer_diffusion_2014,sommer_anisotropic_2015}, it was suggested to describe manifold valued data by a parametric model that generalizes the normal distribution in Euclidean space using the fact that  normal distributions arise as transition distributions of Euclidean Brownian motions. The family of Euclidean Brownian motions can be naturally generalized to Brownian motions on $M$ through stochastic development in the frame bundle. Evaluating at a fixed time yields a distribution on $M$, which we will interpret as a normal distribution. The initial point and infinitesimal covariance of the process may be interpreted as the mean and covariance, respectively.
 
There is a close, though far from trivial or thoroughly understood, relation between distributions on $M$ obtained in this way and the sub-Riemannian distance function in the frame bundle. The non-Euclidean normal distribution will have a density under reasonable assumptions. As we show in this paper, on small time scales, maximizing the log-density corresponds to minimizing the squared sub-Riemannian distance. However, this relationship is weakened as the Brownian motion evolves in time.  

One of the main ideas of \cite{sommer_anisotropic_2015} was that, if we consider our data points as realizations of some  distribution on the manifold, then, rather than minimizing the squared length of paths as in \eqref{squared}, one should maximize the probability of a path. While it is not  a priori clear how this should be understood for a general stochastic variable on $M$, there is a way of defining it for a point that results from a stochastic process. The most probable paths are again closely connected to the sub-Riemannian distance. 

We aim at linking these ideas in this paper. We start in Section \ref{sec2} by discussing the geometry of the frame bundle, in particular the sub-Riemannian metric and the factorization of the frame bundle through the bundle of inner products. We then discuss holonomy, reachable sets, and existence of subbundles where the H\"ormander condition is satisfied.\footnote{Parts of section 2 (holonomy, reachable sets, and the H\"ormander condition) has previously been presented in a conference abstract in \cite{modin_proceedings_2015}.} In Section \ref{sec3}, we recall the Eels-Elworthy-Malliavin construction of Brownian motions in the frame bundle and the associated transition probabilities. In Section \ref{sec4} we identify the most probable paths for the processes. The small time asymptotics of the transition densities in $FM$ and $M$ are described in Section \ref{smalltime}. We conclude the paper by outlining how the theory can be applied in statistics.

\section{Sub-Riemannian Geometry in the Frame Bundle}\label{sec2}
In this section, we recall the definition of the frame bundle. We factor the projection of the frame bundle to the manifold through the bundle of symmetric positive definite covariant 2-tensors, which can be used to represent covariance at points in $M$ in a natural way. Moreover, we define a natural sub-Riemannian structure on the frame bundle and describe its geometry, in particular its holonomy.

\subsection{The Frame Bundle}\label{framedefsec} 
The \emph{frame bundle} $FM$ of a smooth $n$-dimensional manifold $M$ is a smooth vector bundle $FM \xrightarrow{\pi} M$ whose points  consist of a point $x\in M$ and a frame (ordered basis) $(u_1,\dots,u_n)$ for the tangent space $T_xM$. Alternatively, letting $e_1,\dots,e_n$ denote the standard basis in $\R^n$, we may think of a point in $FM$ as an isomorphism $u: \R^n \to T_xM$ where $u(e_i) = u_i$. The frame bundle thus has the structure of a principal bundle with fiber $\mathrm{GL}(n)$.

A \emph{connection} on $FM$ is a smooth splitting of its tangent bundle as $TFM \cong \mathcal{V} \oplus \mathcal{H}$, where $\mathcal{V}$ is the kernel of $\pi_* : TFM \to TM$, referred to as the \emph{vertical} bundle, and $\mathcal{H}$ is any complement, referred to as the \emph{horizontal} bundle. The restriction of $\pi_*$ to $\mathcal{H}_u$ yields a bijection onto $T_{\pi(u)}M$ whose inverse is called \emph{horizontal lifting} and will be denoted $h_u$. 

At any point $u=(u_1,\dots, u_n)\in FM$, there is a  basis $H_1,\dots,H_n$ for $\mathcal{H}_u$, where $H_i = h_u (u_i)$. The $H_i$ define horizontal vector fields. There is a natural correspondance between Euclidean paths $\gamma : [0,1] \to \R^n$ starting at the origin and horizontal paths ${\eta} : [0,1] \to FM$ with fixed initial pont $u_0\in FM$: If $\dot{\gamma}(t) = ( \alpha^1(t),\dots ,  \alpha^n(t))$, then $\eta$  is the  solution to the differential equation $\dot{\eta}(t) =  H_i(\eta(t))\  \alpha^i(t)$ with $\eta(0)=u_0$.
The map $\phi_{u_0}$ taking paths in  $\R^n$ to paths on $FM$ is known as the \emph{development map}.

If $M$ is Riemannian, the Levi-Civita connection naturally defines a horizontal subbundle of $FM$ as follows. Take any curve through $x\in M$ with initial velocity $v$. Parallel transport of the frame vectors $u_1,\dots,u_n\in T_xM$ along the curve yields a path in $FM$ starting at $u$. The derivative of this curve at zero is the horizontal lift $h_u(v)$ of $v$ to $T_uFM$. The set of all horizontal lifts forms the vector bundle $\mathcal{H}$.  Since parallel transport preserves inner products, the horizontal distribution is tangent to the \emph{orthonormal bundle} $OM$ consisting of isometries $u: \R^n \to T_xM$.

\subsection{The Bundle of Symmetric 2-tensors}   
\newcommand{\Symp}{\mathrm{Sym}^+}
Let $\Symp M$ denote the space of symmetric positive definite covariant 2-tensors on $M$. 
There is a natural factorization of the projection $\pi$ as 
\begin{equation*}
FM \xrightarrow{\Sigma} \Symp M \xrightarrow{q} M,
\end{equation*}
where $\Sigma$ maps $u\in FM $ to the inner product 
\begin{equation*}
\Sigma(u)(v,w)=\langle u^{-1}(v), u^{-1}(w)\rangle_{\R^n}, \qquad v,w\in T_{\pi(u)}M.
\end{equation*}	
	
If we choose coordinates $x=(x^1,\dots,x^n)$ on $M$, then there is a set of coordinates on $FM $ given by $x^1,\dots,x^n ,\alpha_i^j$ where $u_i=\alpha_i^j \partial_j$ (here $\partial_j=\frac{\partial}{\partial x^j } $ and we have used the Einstein summation notation). We denote the matrix with $ji$th entry $\alpha_i^j$ by $\alpha$. 
Moreover, there are coordinates on $\Symp M$ given at $\sigma \in \Symp M$  by the point $q(\sigma)=(x^1,\dots,x^n)\in M$ together with the coordinates $\sigma_{ij}=\sigma_{ji}=\sigma(\partial_i ,\partial_j )$. Letting $\beta = \alpha^{-1}$, we have that $u^{-1}(\partial_i)=\beta_i^k e_k$ and thus $\sigma_{ij}(\Sigma(u))=\sum_k \beta_i^k \beta_j^k = (\beta^T \beta)_{ij}$. 
Hence, in local coordinates $\Sigma(x,\alpha)=(x,(\alpha \alpha^T)^{-1})$.

By the polar decomposition  theorem, there is a diffeomorphism between $\Symp M$ and the quotient $ FM/\OO(n)$ where $\OO(n)$ acts on $FM$ from the right by $ (x,\alpha)\rho = (x,  \alpha \rho)$. The coset $(x, \alpha)\OO(n) \in  FM/\OO(n)$  corresponds to $(x,(\alpha \alpha^T)^{-1})\in \Symp M$. 
The horizontal distribution on $FM$ reduces to a distribution $\mathcal{H}^S$, which we also call horizontal, on $\Symp M$.
\begin{proposition}
There is a smooth splitting  $T \Symp M \cong \mathcal{H}^S \oplus \ker q_*$ such that $\Sigma_* \mathcal{H} = \mathcal{H}^S$. In particular, any vector $v\in T_{q(\sigma)} M$ has a unique horizontal  lift to $\mathcal{H}_\sigma^S$ which we denote $h^S_\sigma(v)$. 
\end{proposition}

\begin{proof}
We must show that for any vector $v \in T_xM$ there is a unique vector  $h^S_\sigma(v) \in T_\sigma\Symp M$ such that $\Sigma_*h_u(v) = h^S_\sigma(v)$ for all $u$ with $\Sigma(u)=\sigma$. But this is true because if $u,\tilde{u}\in FM$ with $u_i= \alpha_i^j \partial_j$ and $\tilde{u}_i=( \tilde{\alpha})_i^j \partial_j$ and $\Sigma(u)=\Sigma(\tilde{u})=\sigma$, then there is an orthogonal map $\rho \in \OO(n)$ with $\alpha \cdot \rho = \tilde{\alpha} $. Let $\gamma : [0,1] \to M$ be a path with $\dot{\gamma}(0)=v$ and let $u(t)=\alpha_i^j(t) \partial_j(\gamma(t))$ be a parallel lift with $u(0)=u$. Then $\tilde{u}(t)= u(t) \cdot \rho $ is a parallel lift of $\gamma$ with $\tilde{u}(0)=\tilde{u}$ and $\tilde{\alpha}(t)=\alpha(t) \cdot \rho$. Thus $h_u(v)=\frac{\partial}{\partial t}{u}(t)_{| t=0}$ while $h_{\tilde{u}}(v)=\frac{\partial}{\partial t}{ \tilde{u}}(t)_{| t=0}$.
But then 
\begin{equation*}
\Sigma_*(h_u(v)) =\frac{\partial}{\partial t}  \left(\gamma(t), \left(\alpha(t) \alpha(t)^T\right)^{-1}\right)_{|t=0} = \Sigma_*(h_{\tilde{u}}(v)).
\end{equation*}
By standard principal bundle arguments, the distribution $\mathcal{H}^S$  is smooth.
\end{proof}

\subsection{Sub-Riemannian Structure}
There exists a natural sub-Riemannian metric on the horizontal distribution on the frame bundle. Indeed,
for $v,w\in \mathcal{H}_u$, the inner product is given by
\begin{equation*}
g_u^{FM}(v,w)= \Sigma(u)(\pi_*v , \pi_*w ).
\end{equation*}
Alternatively, we may think of this sub-Riemannian structure as a map
$\tilde{g}^{FM}:\,TFM^*\rightarrow \mathcal{H} \subseteq TFM$  defined by
$
  g_u^{FM}(w,\tilde{g}^{FM}(\xi))
  =   (\xi|w)
$, $
  \forall w\in \mathcal{H}_u$.
The horizontal vector fields $H_i$ are orthonormal with respect to $g_u^{FM}$
\begin{equation*}
  g_u^{FM}(H_i,H_j)
  =
  \Sigma(u)(\pi_*v , \pi_*w )
  =
  \langle u^{-1}(u_i), u^{-1}(u_j)\rangle_{\R^n}
  =
  \langle e_i, e_j\rangle_{\R^n}
\end{equation*}
and they therefore constitute a global orthonormal frame for $\mathcal{H}$.
Note that $g_u$ is different from the induced metric of
bundle type \cite{montgomery_tour_2006}  that makes $\pi_*$ an isometry on $\mathcal{H}$ because $\pi_*g_u^{FM}$ depends on $u$. 

The sub-Riemannian length of an absolutely continuous path $\gamma : [0,1] \to FM$ whose derivative is a.e.\ horizontal is defined by 
\begin{equation*}
l(\gamma) = \int_0^1\sqrt{g_{\gamma(t)}^{FM}(\dot{\gamma}(t),\dot{\gamma}(t))}dt. 
\end{equation*}
If $\dot{\gamma}$ is not a.e.\ horizontal, we set
$l(\gamma)=\infty$. The sub-Riemannian distance between $u_1$ and $u_2$ in $FM$ is then
\begin{equation*}
d_{FM}(u_1,u_2) = \inf\{l(\gamma) \mid \gamma(0)=u_1, \gamma(1)=u_2\}. 
\end{equation*}
As in the Riemannian case, a length minimizing curve is called a \emph{geodesic}. 
Note that the sub-Riemannian distance between two points may be infinite, even if $M$ is connected, because there may be  points that cannot be joined by a horizontal path.

Similarly, there is a sub-Riemannian metric ${g}^{S}$ on $\mathcal{H}^S$   given on the horizontal vectors  $v,w\in T_\sigma \Symp M$ by
\begin{equation*}
{g}^{S}_\sigma(v,w) = \sigma(q_*v,q_*w).
\end{equation*}
This again leads to a notion of sub-Riemannian length and distance $d_{\Symp M}$ on $\Symp M$. Since $g^{FM}(v,w)={g}^{S}(\Sigma_*(v),\Sigma_*(w))$, the map $\Sigma$ takes a horizontal curve on $FM$ to a horizontal curve on $\Symp M$ of the same sub-Riemannian length. This implies that
\begin{equation*}
d_{FM}(u_0, \pi^{-1}(x)) = d_{\Symp M}(\Sigma(u_0),q^{-1}(x)),
\end{equation*}
where $x\in M$ and the distances are the minimal distances to the fibers of $FM$ and $\Symp M$, respectively.

\newcommand{\LieH}{\mathrm{Lie}(\mathcal{H})}
\newcommand{\Lie}{\mathrm{Lie}}

\subsection{The Reachable Set and Holonomy}
Not all points in $FM$ can be joined by a horizontal path. We therefore consider the \emph{reachable set} $Q(u)$ consisting of points that can be reached from $u$ by a horizontal path. For  $u,p\in FM$  we write $u\sim p$ if $u$ and $p$ can be joined by a horizontal curve in $FM$. Then 
\begin{equation*}
Q(u)=\{p\in FM\,|\,p\sim u\}.
\end{equation*}
The reachable set $Q(u)$ is a smooth immersed submanifold \cite{sussmann_orbits_1973}. We let $\piQu$ denote the restriction of $\pi$ to $\Qu$.

 The holonomy group $\mathrm{Hol}_u(FM,\mathcal{H})$ of $\mathcal{H}$ at $u\in FM$ is
\begin{equation*}
  \mathrm{Hol}_u(FM,\mathcal{H})
  =
  \{a\in\GL(T_x M)
    \, |\,
    a \cdot u\sim u
  \},
\end{equation*}
where $ \GL(T_xM) $ is the group of linear automorphisms on $T_xM$ and $a\in \GL(T_xM)$ acts  on the fibers of $FM$ by $ a\cdot u = a\circ u$ in the natural way.  We denote by $\mathrm{Hol}_u^0$ the connected component in $\mathrm{Hol}_u$ containing the identity. 
\begin{proposition}\label{subbundle}
  Let $M$ be Riemannian with corresponding connection $\mathcal{H}$. Fix $u\in FM$. Then $Q(u)$ is a principal subbundle of $FM$ with fibre $\mathrm{Hol}_u(FM,\mathcal{H})$.
\end{proposition}

\begin{proof}
  Theorem 3.2.8 of \cite{joyce_compact_2000} asserts that the holonomy subgroup
  is closed because $M$ is Riemannian. The result then follows from Theorem 2.3.6 
  of \cite{joyce_compact_2000}.
\end{proof}
When $M$ is Riemannian,  the holonomy group
corresponds to the set of frames reachable by parallel transport along loops starting and ending at $\pi(u)$. Parallel transport preserves the Riemannian inner product, and $\mathrm{Hol}_u$ is isomorphic to a subgroup of $\OO(T_xM)$. Thus, if $u$ is orthonormal, then $Q(u)$ is a subbundle of the orthonormal frame bundle $O M$. If $M$ is also orientable, then $\mathrm{Hol}_u$ is a subgroup of $\SO(T_xM)$ because parallel transport preserves orientation. 

\subsection{The H\"{o}rmander Condition}
\label{sec:hormander}
A distribution $\mathcal{D}$ on a manifold $N$ is said to satisfy the \emph{H\"{o}rmander condition} if
it is \emph{bracket generating}, i.e.\ if $\Lie(\mathcal{D}) = TN$,
where $\Lie(\mathcal{D})$ denotes the Lie algebra generated by $\mathcal{D}$. 

Recall the decomposition  $T_uFM \cong \mathcal{H}_u\oplus \mathcal{V}_u$. The Lie algebra $\mathfrak{h}_u$ of $\mathrm{Hol}_u$ must be contained in  
$\mathcal{V}_u$. Moreover, the vertical part of $\LieH_u$ must be contained in $\mathfrak{h}_u$. 
Under the condition of Proposition \ref{subbundle}, $\LieH_u \subseteq  \mathfrak{h}_u\oplus \mathcal{H}_u = T_uQ(u)$, so any connected submanifold of $FM$ on which the H\"{o}rmander condition holds must be contained in $Q(u)$. There are several reasons why we are interested in whether the H\"{o}rmander condition satisfied on $Q(u_0)$. Some are listed in the next proposition.
\begin{proposition}\label{metricprop}
Suppose $\mathcal{H}$ satisfies the H\"{o}rmander condition on $Q(u_0)$. If $Q(u_0)$ has closed fibers, then $Q(u_0)$ is a closed subbundle of $FM$. In this case
\begin{itemize}
\item[(i)] $d_{FM}$ is continuous on $Q(u_0)$.
\item[(ii)] $Q(u_0)$ with the metric $d_{FM}$ is complete as a metric space.
\item[(iii)] any two points in $Q(u_0)$ can be joined by a minimal geodesic.
\end{itemize}
\end{proposition}
\begin{proof}
The first claim is \cite[Thm 2.3.6]{joyce_compact_2000}. (i) follows for instance from the ball-box theorem \cite[Thm 2.10]{montgomery_tour_2006}. (ii) follows because we can extend the sub-Riemnnian metric to a metric on all of $Q(u_0)$. Finally (ii) implies (iii) according to \cite[Lem D.9]{montgomery_tour_2006}.
\end{proof}
Below, we give conditions under which the H\"{o}rmander condition is indeed satisfied on $Q(u)$.
The
discussion is based on
\cite{bryant_survey_1987,joyce_compact_2000,montgomery_tour_2006,sussmann_orbits_1973}.

Suppose $M$ is Riemannian. Injectivity of the curvature tensor $R_x:\Lambda^2(T_xM)\rightarrow \mathfrak{so}(T_xM)$ implies  surjectivity due to the dimensions of $\Lambda^2(T_xM)$ and $\mathfrak{so}(T_xM)$. 
In this situation, the H\"ormander condition is
satisfied on the subbundle $Q(u)$. Such injective curvature metrics 
are generic, i.e. they form an open and dense subset of all metrics on $M$
\cite{bryant_survey_1987}.  

\begin{theorem}
  If $M$ is Riemannian and the curvature map is injective in every point, then the horizontal distribution is bracket generating on $Q(u)$ and $\mathrm{Hol}_u^0=\mathrm{SO}(T_xM)$.
  \label{thm:surjective-bracket-generating}
\end{theorem}

\begin{proof}
Since $\LieH_u \subseteq T_uQ(u) \subseteq \mathcal{H}_u \oplus \mathfrak{so}(T_xM)$, it suffices to show that $\LieH_u = \mathcal{H}_u \oplus \mathfrak{so}(T_xM)$. For this, it is enough that the span of $\mathcal{H}$ and its first bracket equals $\mathcal{H}\oplus \mathfrak{so}(T_xM)$,
  i.e. $\mathcal{H}\oplus [\mathcal{H},\mathcal{H}]=\mathcal{H}\oplus \mathfrak{so}(T_xM)$. Thus, let 
  $z=z_v+z_h \in \mathfrak{so}(T_xM)\oplus \mathcal{H}_u$. 
 By assumption, $R$ is surjective onto $\mathfrak{so}(T_xM)$  so we can find horizontal vector fields
  $V,W$ such that $R(V,W)=z_v$. Since
  $R(V,W)=[V,W]-h_u([\pi_*(V),\pi_*(V)])$, we have 
  $z=z_v+z_h=[V,W]-h_u([\pi_*(V),\pi_*(V)])+z_h$. The first term lies in 
  $[\mathcal{H},\mathcal{H}]$ and the two last terms belong to $\mathcal{H}$.  
\end{proof}

When $R$ is not injective, it is still possible that $Q(u)$ satisfies H\"{o}rmander's condition in some non-degenerate situations:
\begin{theorem}
  If $ \LieH_u$ has constant rank for all $u\in FM$, then $Q(u)$ satisfies the H\"ormander condition.
    \label{thm:constant-dim-generating}
\end{theorem}

The constant rank condition is for instance satisfied for analytic manifolds \cite[Appendix C]{montgomery_tour_2006}  and homogeneous spaces.

\begin{proof}
  The distribution $\LieH$ is involutive by definition. The constant rank ensures that the Frobenius theorem (see \cite{michor_topics_2008} Theorem 3.20) applies. Thus for any $u\in FM$ there exists a maximal connected immersed submanifold $Q_{\LieH}(u)$ containing $u$ of dimension $\dim \LieH$ with tangent space $\LieH$. By construction, the H\"ormander condition is satisfied on $Q_{\LieH}(u)$.

Chow's theorem \cite[Theorem 2.2]{montgomery_tour_2006} yields that $Q_{\LieH}(u)
\subseteq Q(u)$. On the other hand, any two points in $Q(u)$ can be joined by a
horizontal curve. By the construction of $Q_{\LieH}(u)$, this curve must lie in
$Q_{\LieH}(u)$. We deduce that $Q(u)$ and $Q_{\LieH}(u)$ are equal as sets. Moreover, it follows from \cite[Exercise C.4]{montgomery_tour_2006}, see also \cite{sussmann_orbits_1973}, that $\dim \LieH_u =
\dim \mathrm{Hol}_u$, so $Q(u)=Q_{\LieH}(u)$ as differentiable
manifolds.
\end{proof}

In general, however, $\LieH$ may not have constant dimension. In this case, it is not possible to find a submanifold of $FM$ where $\mathcal{H}$ is bracket generating. For instance, if $M$ is flat in a neighborhood of $\pi(u)$ then $\dim \LieH_u = n$, while the dimension  of $\LieH$ will be larger in curved parts of $M$.

\subsection{The Hamilton-Jacobi Equations}
A class of geodesics for the sub-Riemannian metric on $FM$, called the \emph{normal geodesics} can be computed by solving an ordinary differential equation on $T^*Q(u_0)$. Recall that the sub-Riemannian metric can be written as a cometric $\tilde{g}^{FM}:\,T^*Q(u_0)  \rightarrow \mathcal{H} \subset TQ(u_0)$. 
For $q\in Q(u_0)$ and $p\in T^*_qQ(u_0)$, we define the Hamiltonian  
\begin{equation*}
H(q,p)=\frac{1}{2} p(\tilde{g}^{FM}(p)).
\end{equation*}
The Hamiltonian differential equations in canonical coordinates $(q_i,p_i)$ on $T^*Q(u_0)$ look as follows
\begin{equation*}
\dot{q}^i= \frac{\partial H}{\partial p_i}(q,p), \qquad \dot{p}_i=- \frac{\partial H}{\partial q^i}(q,p).
\end{equation*}
If the H\"{o}rmander condition is satisfied on $Q(u_0)$, then the projection of any solution is a (locally) length minimizing geodesic for the sub-Riemannian structure, called a \emph{normal geodesic}. 
However, contrary to ordinary Riemannian geometry, there exist length minimizing
geodesics that do not arise in this way.  See e.g.\ \cite{montgomery_tour_2006} for details.

There is a sub-Riemannian version of the exponential map 
\begin{equation*}
\exp_{u_0}: T^*_{u_0} Q(u_0)\cap H^{-1}\left(\tfrac{1}{2}\right) \to Q(u_0) 
\end{equation*}
that takes an initial covector $p_0\in T^*_{u_0}Q(u_0)$ to the endpoint of the geodesic that results from projecting the solution to the Hamiltonian equations with initial condition $(u_0,p_0)$. However, the exponential map is not a local diffeomorphism around 0 and may not be surjective.

In a similar way, if we choose to work with geodesics on $\Symp M$,  there are Hamiltonian equations for the normal geodesics. This has the practical advantage of reducing the number of equations. Again, there is no guarantee that all geodesics are normal.
An expression in local coordinates of the sub-Riemannian Hamiltonian equations on $FM$ for computational purposes can be found in \cite{sommer_evolution_2015}.

\section{Brownian Motion and Stochastic Development}\label{sec3}
The Brownian motion $W_t$ on Euclidean space starting from $\mu \in \R^n$ and having covariance matrix $\Sigma$ is an almost surely continuous stochastic process with stationary independent increments whose distribution at time $t$ satisfies 
\begin{equation*}
W_t \sim N_n(\mu, t\Sigma)\ . 
\end{equation*}
The Brownian motion can be generalized to any manifold $M$ with connection through the process of \emph{stochastic development} that allows any Euclidean semi-martingale to be mapped uniquely to a corresponding process in $FM$ given a starting point $u_0\in FM$ \cite{hsu_stochastic_2002}. Developing the standard Brownian motion in $FM$ and projecting to $M$ yields a stochastic process on $M$. 
Evaluated at a fixed time $t$, its distribution can be regarded a generalization of the normal distribution to $M$.
If $u_0$ is orthonormal, the resulting process is what is normally called the Brownian motion on $M$. If $u_0$ is non-orthonormal, the resulting distribution can be viewed as a Brownian motion on $M$ with anisotropic covariance matrix \cite{sommer_anisotropic_2015}. 

Below, we make this construction concrete, we discuss the reduction to $\Symp M$, where the distribution on $M$ depends uniquely on the initial value, and the restriction to the subbundles  $Q(u_0)$ where the transition densities are smooth.

\subsection{Brownian Motion in the Frame Bundle}
Given a manifold $M$ with connection, there is a stochastic version of the development map from Section \ref{framedefsec}. 
Given a stochastic process $W_t$ in $\mathbb{R}^n$ starting at $0$, the
stochastic development of $W_t$ in $FM$ starting at $u_0\in FM$ is a 
stochastic process $U_t$ on $FM$ solving the Stratonovich stochastic
differential equation 
\begin{equation}\label{mainSDE}
dU_t=H_i(U_t)\circ dW_t^i
\end{equation}
with initial condition $U_0=u_0\in FM$, see e.g. \cite{hsu_stochastic_2002}.
The process $U_t$ is a diffusion on $FM$, and the projection $X_t=\pi U_t$  onto $M$ yields a stochastic process on $M$.
 We denote the stochastic development map $W_t \mapsto X_t$ by $\tilde{\phi}_{u_0}$.

Let $M$ be equipped with the Riemannian metric $g$. Since parallel transport
preserves the Riemannian inner product, the horizontal vector fields are tangent
to the submanifolds of $FM$ where the coordinates $g(u_i,u_j)$ of the metric $g$ in the
frame $u$ are constantly equal those of $u_0$. It follows that any solution to  $dU_t = H_i(U_t)\circ dW^i_t$ with initial condition $u_0$ will stay inside these submanifolds
according to \cite[7.23]{emery_stochastic_1989}.
In particular, if $u_0$ is orthonormal, then  $g_{u_0}(u_i,u_j)=\delta_{ij}$ and hence $U_t$ will stay inside the orthonormal bundle $OM$.

If $u_0$ is orthonormal and $W_t$ is a standard Brownian motion, then $X_t$ is the usual Brownian motion on $M$ \cite{hsu_stochastic_2002}.
If $W_t$ is a standard Brownian motion but $u_0$ is not orthonormal, then we may think of $X_t$ as a generalized Brownian motion on $M$. Hence, the distribution of $X_t$ at time $t=1$ can be thought of as the analogue of a normal distribution on $M$. The starting point $x_0 = \pi(u_0)$ can naturally be interpreted as the mean of the distribution, while the symmetric bilinear positive definite map $\Sigma(u_0)$ on $T_{x_0}M$ corresponds to the precision matrix, i.e. the inverse of the covariance matrix. The covariance will be anisotropic if the starting frame $u_0$ is not orthonormal.  As we shall see below, the distribution of $X_t$ is uniquely determined by $\Sigma(u_0)$. 

\subsection{Reduction of Initial Conditions}
The $n$ horizontal vector fields $H_i$ on $FM$ define a second order differential operator on $FM$ by 
\begin{equation}
L=\frac{1}{2}\sum_{i=1}^n H_i^2. 
\end{equation}
The Brownian motion $U_t$ on $FM$ is an $L$-diffusion in the sense of \cite{hsu_stochastic_2002}, meaning that for any smooth $f: FM \to \R$,
\begin{equation}\label{Ldiff}
f(U_t) - f(U_0) - \int_0^t L(f)(U_s) ds  
\end{equation}
is a local martingale.

The restriction of $L$  to $OM$ descends to  the Laplace-Beltrami operator $\Delta_g$ on $M$, i.e. if $f:M\to \R$ is a smooth function then $L(f\circ \pi) = \Delta_g(f)$. Hence the standard Brownian motion can be intrinsically defined on $M$ as a $\Delta_g$-diffusion.  In general, $L$ does not reduce to an operator on $M$, so the generalized Brownian motions are not intrinsically defined. We can, however, reduce it to an operator on $\Symp M$. 
The map $\Sigma : FM \to \Symp M$ maps $U_t$ to a stochastic process $Y_t$ on $\Symp M$.  We will show that there exists a (degenerate) elliptic operator $\tilde{L}$ on $\Symp M$ such that $Y_t$ is an $\tilde{L}$-diffusion.

\begin{proposition}
There is a (degenerate) elliptic operator $\tilde{L}$ on $\Symp M$ that satisfies $L(f\circ \Sigma) =\tilde{L}(f)$ for any function $f$ on $\Symp M$. In particular, $Y_t$ is an $\tilde{L}$-diffusion.
\end{proposition}

\begin{proof}
Choose local coordinates on $M$. Then there are horizontal  vector fields $ h^S(\partial_i )$ on $\Symp M$. Consider
\begin{align*}
\sum_i H_i^2 (f\circ \Sigma)(u) {}& = \sum_i  \left(h_u\left(\alpha_i^j(u) \partial_j\right)\right)^2 (f\circ \Sigma)(u) \\
&=\sum_i \alpha_i^j(u)\, \alpha_i^k(u)\, h_{u}\left(\partial_j\right)\left(h^S_{\Sigma(u)}(\partial_k)f\right)(\Sigma(u))\\
&\quad + \sum_i \alpha_i^j(u) \, \left(h_u(\partial_j)\, \alpha_i^k(u)\right) \left(h^S_{\Sigma(u)}(\partial_k)f\right)(\Sigma(u))\\
& =\sum_i \sigma^{jk}\left(\Sigma(u)\right) h^S_{\Sigma(u)}\left(\partial_j\right)h^S_{\Sigma(u)}\left(\partial_k\right)(f)\left(\Sigma(u)\right) \\
&\quad + \sum_i \alpha_i^j(u)  \left(h_u\left( \partial_j\right) \alpha_i^k(u)\right) \left(h^S_{\Sigma(u)}(\partial_k)f\right)\left(\Sigma(u)\right)
\end{align*}
Here $\sigma^{jk}=\sigma^{-1}(\partial_j,\partial_k )$. Note that the function $\eta^k(u)=\sum_i \alpha_i^j(u)\, \left( h_u(\partial_j)\, \alpha_i^k(u)\right)$ depends only on $\Sigma(u)$ since if $\rho$ is an orthogonal map, then 
\begin{equation*}
\eta_i^k( u\cdot \rho) = \sum_i \rho_{i}^s\, \alpha_s^j(u) \, h_{u\cdot \rho}(\partial_j) \, \alpha_l^k( u \cdot \rho ) = \sum_i\rho_{i}^s\, \rho_i^l\, \alpha_s^j(u)\, h_{u}(\partial_j) \, \alpha_l^k(u)=\eta_i^k(u). 
\end{equation*}
Letting $V_i(\sigma) = (\sigma^{1/2})^{jk}h^S_{\sigma}(\partial_k )$, we thus find that the theorem holds with
\begin{equation*}
\tilde{L}= \sum_i V_i^2 + \eta^k h^S(\partial_k ).
\end{equation*}
The last statement follows from the definition of an $L$-diffusion \eqref{Ldiff} and the corresponding definition of an $\tilde{L}$-diffusion.
\end{proof}

As a consequence, we find that the distribution of $X_t$ is overparametrized by the initial conditions in $FM$.

\begin{corollary}\label{reduce}
The distribution of $Y_t$ (and hence $X_t$) depends only on $\Sigma(U_0)\in \Symp M$.
\end{corollary}

\begin{proof}
According to \cite[Thm. 1.3.6]{hsu_stochastic_2002} an $\tilde{L}$-diffusion is uniquely determined by its initial distribution.
\end{proof}

We now show that the distribution of $X_t$ depends only on $\Sigma(U_0)$. To see this, let $a \in \GL(n)$ be a linear map. The action  $u\mapsto u\cdot a$ on $FM$ defines a diffeomorphism $FM \to FM$. Let $V_t = U_t \cdot a$. This
is a new process with $\pi V_t = \pi U_t =X_t$. Using \cite[Prop.
1.2.4]{hsu_stochastic_2002} and that $a_*H_i(u) = h_{u\cdot a}( u_i)=(a^{-1})^j_i H_j(u\cdot a)$, we see that $V_t $ solves the SDE 
\begin{equation}\label{orthogonalsde}
dV_t= (a^{-1})_i^j H_j(V_t)  \circ dW_t^i = H_j(V_t)\circ d(a^{-1}W)_t^j 
\end{equation}
with initial condition $V_0 = u_0 \cdot a$. But $a^{-1}W_t$ is a Brownian motion on $\R^n$ with covariance matrix $(a^Ta)^{-1}$. Choosing $a$ to be orthogonal yields another proof of Corollary \ref{reduce}. 

\begin{proposition}
The distribution of the process $X_t=\pi(U_t)$, where $U_t$ is a solution of \eqref{mainSDE} with initial condition $U_0=u_0$, is uniquely determined by $\Sigma(u_0)$.
\end{proposition}

\begin{proof}
Suppose that we are given two processes $U_t^1$ and ${U}_t^2$ each solving $dU_t^\nu= H_i(U_t^\nu)\circ dW_t^i$, $\nu=1,2$, and with $U_0^1 \cdot a^1= {U}_0^2 \cdot a^2$ for some $a^1,a^2\in \GL(n)$. Define $V_t^\nu= U_t^\nu \cdot a^\nu$. Then both $V_t^\nu$ solve \eqref{orthogonalsde} and satisfy $V_0^1=V_0^2$. 
The anti-development theorem \cite[Theorem 2.3.4 +
2.3.5]{hsu_stochastic_2002} shows that the process $(a^\nu)^{-1}W_t$ can be recovered uniquely by solving two stochastic differential equations involving only $\pi V_t^\nu = \pi U_t^\nu =X_t^\nu$. Since the distribution of $a^{-1}W_t$  is uniquely determined by $(a^Ta)^{-1}$, $X_t^1$ and $X_t^2$ can only have the same distribution if $((a^1)^Ta^1)^{-1}=((a^2)^Ta^2)^{-1}$. 
\end{proof}

\subsection{Transition Probabilities}\label{transition}
We first recall that a finite set of vector fields $X_i$ on a manifold define a \emph{sub-Laplacian}  with respect to a given volume form $\mu$ on the manifold by the formula
\begin{equation*}
\Delta=\sum_i X_i^2 + \text{div}_{\mu}(X_i) X_i,
\end{equation*}
where the divergence is computed with respect to $\mu$.

\begin{proposition}\label{sublaplacian}
Let $M$ be Riemannian and suppose that $\mathcal{H}$ satisfies the H\"{o}rmander condition on $Q(u_0)$.
Then there exists a volume form $\mu_Q$ on $Q(u_0)$ such that the horizontal vector fields $H_i$ have divergence 0. In particular, $L$ is a sub-Laplacian with respect to this volume form.
\end{proposition}

\begin{proof}
According to Proposition \ref{subbundle},  $Q(u_0)$ is a subbundle of $FM$. In \cite{mok_differential_1978} a Riemannian metric $G$ is defined on $FM$ such that the $H_i$ are divergence free with respect to $G$ (see \cite[Thm
8.4]{mok_differential_1978}). This metric makes the horizontal and vertical distributions orthonormal. Let $\tilde{G}$ be the restriction of $G$ to $Q(u_0)$. Moreover, let $\nabla $ and $\tilde{\nabla}$ be the Levi-Civita connections with respect to $G$ and $\tilde{G}$, respectively.
Locally choose orthonormal vector fields $E_1,\dots, E_n$ spanning $\mathcal{H}$ and extend by $V_1,\dots,V_d$ to an orthonormal basis of $TQ(u_0)$ and further extend by $W_1,\dots, W_{n^2-d}$ to an orthonormal basis for $TFM$. The divergence with respect to the volume form associated to $G$ is given by the trace of the connection, so
\begin{align*}
  \text{div}_{G} H_i = \sum_{j=1}^n \langle \nabla_{E_j}H_i,E_j \rangle_G+
  \sum_{j=1}^d \langle \nabla_{V_j}H_i,V_j \rangle_G+\sum_{j=1}^{n^2-d} \langle \nabla_{W_j}H_i,W_j \rangle_G.
\end{align*}
Since $\tilde{\nabla}$ is given by projecting $\nabla$ onto $TQ(u_0)$ and since $\nabla$ is compatible with the metric,
\begin{align*}
\text{div}_{G} H_i{}& = \sum \langle \tilde{\nabla}_{E_i}H_i,E_i \rangle_G+ \sum \langle \tilde{\nabla}_{V_i}H_i,V_i \rangle_G+\sum \langle \nabla_{W_i}H_i,W_i \rangle_G\\
&= \text{div}_{\tilde{G}} H_i +\sum \langle H_i, \nabla_{W_i} W_i \rangle_G.
\end{align*}
According to \cite[Thm 4.3]{mok_differential_1978}, the fibers of $FM$ are autoparallel, hence $\nabla_{W_i} W_i$ is vertical and
\begin{equation*}
0=\text{div}_{G} H_i = \text{div}_{\tilde{G}} H_i.
\end{equation*}
\end{proof}
Still assuming $M$ to be Riemannian, let $U_0$ be a Brownian motion on $FM$ with initial value $u_0$ and assume that the horizontal vector fields satisfy H\"{o}rmander's condition on $Q(u_0)$. Since $L$ is a sub-Laplacian by Proposition \ref{sublaplacian}, there exists a unique solution $p_t^{Q(u_0)}(u_1,u_2)$ to the sub-Riemannian heat equation
\begin{equation*}
\frac{\partial}{\partial t} p_t^{Q(u_0)}(u_1,u_2) = L p_t^{Q(u_0)}(u_1,u_2)
\end{equation*}
such that $\lim_{t\to 0} p_t^{Q(u_0)}(u_1,u_2)=\delta_{u_1}(u_2)$ in the sense of distributions and  $p_t^{Q(u_0)}$ is smooth on $(0,\infty)\times Q(u_0) \times Q(u_0)$ \cite{strichartz_sub-riemannian_1986}. The function $p_t^{Q(u_0)}$ is called the heat kernel and provides a density with respect to $\mu_Q$ for the distribution of $U_t$ at time $t$
(see \cite{hsu_stochastic_2002} for an argument in the Riemannian case).

Let $\piQuo$ denote the restriction of $\pi$ to $Q(u_0)$. The distribution of $X_t$ satisfies 
\begin{align}\label{density}
  P(X_t\in C){}& = P(U_t \in \piQuo^{-1}(C))\\ \nonumber
   &= \int_{\piQuo^{-1}(C)} p^{Q}_t(u_0,u)\mu_Q(du)\\ \nonumber
   &= \int_C \int_{\piQuo^{-1}(y)}p^{Q}_t(u_0,u) \mu_{\piQuo^{-1}(y)}(du) \mu_g(dy)
\end{align}
for any Borel set $C\subseteq M$. Here we have used that the volume form $\mu_Q$ defined in the proof of Proposition \ref{sublaplacian} satisfies $d\mu_Q=d\mu_{\piQuo^{-1}(y)} d\mu_g$ where $\mu_g$ is the Riemannian volume form on $M$ associated with $g$ and $\mu_{\piQuo^{-1}(y)}$ is the volume form on $\piQuo^{-1}(y)$ coming from regarding $\piQuo^{-1}(y)$ as a subgroup of $\GL(n)$ and restricting the Euclidean metric. This follows from the
expression of the metric $G$ in \cite[Thm 4.1]{mok_differential_1978}.
Equation \eqref{density} shows that $X_t$ has a density, and since this only depends on $\sigma=\Sigma(u_0) \in \Symp M$, we denote it by $p_t^M(\sigma, y)$. 
We see that $p_t^M$ is given by
\begin{equation}
   p_t^M(\Sigma(u_0),y)
  =   
  \int_{\piQuo^{-1}(y)} p^Q_t(u_0,u)du.
  \label{eq:ptm}
\end{equation}

The density $p_t^M$ can be taken as a starting point for a statistical estimation procedure, such as maximum likelihood estimation, in order to estimate the initial conditions $\sigma \in \Symp M$. There is no closed formula for $p_t^M$, not even in the well-studied isotropic case (see \cite{hsu_stochastic_2002} for asymptotic results). In Section \ref{smalltime}, we derive asymptotic results that can be used for heuristic estimation procedures, see Section \ref{statistics}.

\section{Most Probable Paths}\label{sec4}
Instead of using the geodesic distance to measure the similarity between points, it was argued in Section \ref{background} that when the data points are considered outcomes of a stochastic process on $M$, it is natural to consider the maximal probability of a path connecting two points. We are going to describe the notion of most probable paths and their characterization via the Onsager-Machlup functional \cite{fujita_onsager-machlup_1982}. Afterwards, we discuss how most probable paths for the driving Euclidean process relate to the sub-Riemannian metric. This does not require $M$ to be Riemannian.

If $X_t$ is a stochastic process on a manifold $M$ with initial point $X_0=x_0$, the most probable path from $x_0\in M$ to another point  $y\in M$ can be defined as the path that maximizes the probability that $X_t$ sojourns around the path \cite{fujita_onsager-machlup_1982}.
 More formally, let $\gamma: [0,1]\to M$ be a smooth path with $\gamma(0) = x_0$. We denote by $\mu_\epsilon^M(\gamma)$ the probability that $X_t$ stays within distance $\epsilon$ from $\gamma$, i.e.
$$
\mu_\epsilon^M(\gamma)
=
P\left(
d_g(X_t,\gamma(t))<\epsilon,
\ \forall t\in[0,1]
\right)
\ .
$$
The most probable path is the path that maximizes $\mu_\epsilon^M(\gamma)$ when $\epsilon \to 0$

\subsection{The Isotropic Case}
The most probable paths of a Brownian motion have only been determined in the case where the Brownian motion is isotropic, i.e.\ the starting frame $u_0$ is orthonormal. The reason why this case is easier to handle is that the isotropic Brownian motion is intrinsically defined on $M$ as the diffusion generated by the Laplace-Beltrami operator.    

\begin{theorem}[see e.g.\ \cite{fujita_onsager-machlup_1982}]
Let $X_t$ be an isotropic Brownian motion on $M$  and let $\gamma: [0,1]\to M$ be a smooth path starting at $x_0$.
Then as $\epsilon \to 0$
\begin{equation}\label{OMeq}
\mu_\epsilon^{M}(\gamma)
=
\exp\left(c_1+\frac{c_2}{\epsilon^2}+
\int_0^1L_M(\gamma(t),\dot{\gamma}(t))dt\right),
\end{equation}
where $c_1,c_2$ are constants independent of $\gamma$, $L_M$ is the Onsager-Machlup functional
$$
L_M(\gamma(t),\dot{\gamma}(t))
=
-\frac{1}{2}\|\dot{\gamma}(t)\|^2_g
+
\frac{1}{12}S(\gamma(t)),
\ 
$$
and $S$ is the scalar curvature on $M$. 
\end{theorem}
The most probable paths are thus the paths that maximize the Onsager-Machlup functional 
\begin{equation}\label{energy}
\int_0^1 L_M(\gamma(t),\dot{\gamma}(t))dt = -E(\gamma)+ \frac{1}{12}\int_0^1S(\gamma(t))dt 
\end{equation}
involving the energy $E(\gamma)=\frac{1}{2}\int_0^1\|\dot{\gamma}(t)\|^2_g dt$ of the path plus a curvature correction term. In comparison, geodesics are the paths that simply minimize the energy. Hence, the most probable paths are generally not geodesics. The Onsager-Machlup functional thus provides a way of measuring similarity between points  different from the geodesic distance. 
However, in spaces of constant scalar curvature, for instance homogeneous spaces, the most probable paths coincide with geodesics. In particular, this holds in Euclidean space.
\begin{corollary}
Let $W_t$ be an isotropic Brownian motion in $\R^n$ starting at 0. Let $\gamma:[0,1]\to \R^n$ be a smooth path with $\gamma(0)=0$. Then $-L_{\R^n}(\gamma(t),\dot{\gamma}(t))= \frac{1}{2}\|\dot{\gamma}(t)\|_{\R^n}^2$. Therefore, the most probable path between two points is the straight line.
\label{lem:Lphi}
\end{corollary}

\subsection{Anisotropic Case}
There seems to be no analogue of the Onsager-Machlup theorem available in the literature for Brownian motions with anisotropic covariance. With the construction of anisotropic processes that is the focus of this paper, an Onsager-Machlup theorem would apply to a process in the horizontally connected component $Q(u_0)$ in the frame bundle. Several problems occur, many of which are due to the sub-Riemannian structure. For instance, it is not clear whether to use a fiber bundle metric or the sub-Riemannian metric to define tubes. Another problem is the lack of normal coordinates that are essential in all proofs of the Onsager-Machlup theorem.

Instead, we model the most  probable paths of the \emph{driving process}, i.e.\ the underlying Euclidean Brownian motion $W_t$. Another advantage of this is that it does not require a Riemannian structure on $M$ but only a connection. Recall that the stochastic development  map $\tilde{\phi}_{u_0}$ takes the paths of a Euclidean Brownian motion $W_t$ with $W_0=0$ to the paths of a Brownian motion $X_t$ on $M$ by projecting a path $U_t$ in the frame bundle with $U_0=u_0$. However, it doesn't directly provide a link between the most probable paths in $\R^n$ and the most probable paths in $M$ that maximize the Onsager-Machlup function because the development map does not preserve tubes: if the Euclidean Brownian motion stays close to a path $\gamma$, the stochastic development of the path does not necessarily stay close to the development of $\gamma$ in $M$. 

\begin{definition}
Let $W_t$ be a standard Brownian motion and let $X_t=\phi_{u_0}(W_t)$  be a Brownian motion on $M$.
A \emph{most probable path for the driving process} from $x_0=\pi(u_0)\in M$ to $y\in M$ is a smooth path $\gamma: [0,1]\to M $ with $\gamma(0)=x_0 $ and $\gamma(1) = y$ such that its anti-development $\phi^{-1}_{u_0}(\gamma)$ is the most probable such path for $W_t$. That is, $\gamma$ is given by
\begin{equation*}
\textrm{argmin}_{\gamma, \gamma(0)=x_0, \gamma(1)=y} \int_0^1 -L_{\R^n}(\phi_{u_0}^{-1}(\gamma)(t),\tfrac{d}{dt}\phi_{u_0}^{-1}(\gamma)(t))\, dt.
\end{equation*}
\end{definition}
From the Onsager-Machlup theorem in Euclidean space, we obtain a characterization of the most probable paths for the driving process. 
The theorem shows that the most probable paths for the driving process coincide with the most probable paths introduced formally in \cite{sommer_anisotropic_2015}.
\begin{theorem}\label{driving}
Let $u_0$ be a frame in $T_{x_0} M$, let $W_t$ be a standard Brownian motion, and let
 $X_t=\tilde{\phi}_{u_0}(W_t)$. Suppose the H\"{o}rmander condition is satisfied on $Q(u_0)$ and that $Q(u_0)$ has compact fibers. Then most probable paths from $x_0$ to $y\in M$ for the driving process of $X_t$ exist, and they are projections of sub-Riemannian geodesics in $FM$ minimizing the sub-Riemannian distance from $u_0$ to $\pi^{-1}(y)$.
\end{theorem}

\begin{proof}
Let $\zeta $ be a smooth path in $\R^n$ and consider the development $\eta$ of $\zeta$ in $FM$ such that 
  $$
 \dot{\eta}_t
  =H_i({\eta}_t)\dot{\zeta}^i_t.
  $$
  Since the $H_i$ are orthonormal in the sub-Riemannian metric $g^{FM}$,
  \begin{align*}
  \|\dot{\eta}\|_{FM}^2
  &=   \| \dot{\zeta}_t\|_{\R^n}^2
  =   -L(\zeta_t).
  \end{align*}
In the following, $\tilde{\gamma}$ denotes the horizontal lift of a path $\gamma$ on $M$ with $\tilde{\gamma}(0)=u_0$. The most probable path  for the driving process are then given as
  \begin{align*}
&\mathrm{argmin}_{\gamma, \gamma(0)=x_0, \gamma(1)=y }  -\int_0^1 L( \phi_{u_0}^{-1}(\gamma)(t),\tfrac{d}{dt}\phi_{u_0}^{-1}(\gamma)(t)) dt
\\
&=
\mathrm{argmin}_{\gamma, \gamma(0)=x_0,\, \gamma(1)=y } \int_0^1 \|\dot{\tilde{\gamma}}(t)\|_{g}^2 dt\\
&=
\mathrm{argmin}_{\gamma, \gamma(0)=x_0,\, \gamma(1)=y } \, l (\tilde{\gamma}) 
\ .
  \end{align*}
 By compactness of $\pi^{-1}(y)$ and continuity of $d_{FM}$ on $Q(u_0)$ (see Proposition \ref{metricprop}), there exists a horizontal path $\tilde{\gamma}$ in $FM$ that minimizes the latter expression. This proves the claim.
\end{proof}

In the case where $M$ is Riemannian and $u_0$ is orthonormal, the most probable paths for the driving process are precisely the geodesics. The above theorem thus shows that, in general, the most probable paths for the driving process are not the same as the most probable paths for $X_t$, since the scalar curvature term in \eqref{energy} does not appear in Theorem \ref{driving}. See \cite{sommer_anisotropic_2015} for visual illustrations of most probable paths for the driving process on $\mathbb{S}^2$ with varying degrees of anisotropy in $u_0$.

\section{Small Time Asymptotics}
\label{smalltime}

We assume throughout this section that $M$ is Riemannian and that the H\"{o}rmander condition is satisfied on $Q(u_0)$. Let $m$ denote the dimension of $Q(u_0)$. As argued in Section \ref{transition}, the distribution of $U_t$ has a smooth density $p_t^{Q(u_0)}$. Since the sub-Riemannan structure on $Q(u_0)$ is complete by Proposition \ref{metricprop}, the density exhibits the small time limit \cite{barilari_small-time_2012,varadhan_behavior_1967}
\begin{equation*}
  \lim_{t\rightarrow 0}
  2t\log p_t^{\Quo}(y)
  =
  -\dQuo(u_0,y)^2.
\end{equation*}
A point $u\in\Qo$ is called a \emph{smooth point} if there exists $\xi\in T_{u_0}^*\Qo$ such
that $\exp_{u_0}(\xi)=u$, $\xi$ is a regular point of $\exp_{u_0}$, and the resulting geodesic
connecting $u_0$ and $u$ is unique and length minimizing. 
The set of smooth points is open and dense \cite{agrachev_any_2009} but it may not be of measure zero.

For the standard Brownian motion on $M$, the refined statement
\begin{equation*}
  \lim_{t\rightarrow 0}
  2t\log p_t^{M}(x)
  =
  -d_g(x_0,x)^2.
\end{equation*}
holds. In this case, we can 
interpret the small time limit as the transition density
approaching the density of an isotropic diffusion in the linear tangent space
$T_xM$ because $d_g(u_0,y)^2=\|\Log_{u_0}y\|_g^2$. In the sub-Riemannian case, making the
same analogy is complicated by the fact that the exponential map is not a local diffeomorphism in a small neighborhood of ${u_0}$. However,  a smooth point ${u}$ is the
image of $ {\xi}$ under $\exp_{u_0}$ where $\xi$ is a regular point, so the $\Log$-map is defined in a neighborhood of ${u}$. Thus we locally have an interpretation similar to the Riemannian case.

When $y$ is a smooth point, there is a more explicit expansion
\begin{equation}
  p_t^{\Quo}(u_0,y)
  =
  (2\pi t)^{-m/2}
  e^{-\frac{\dQuo(u_0,y)^2}{2t}}
  v(u_0,y,t),
  \label{eq:subriemanniansmalltime}
\end{equation}
where $v(u_0,y,t)$ is smooth for  $(y,t)\in M \times (0,\infty)$ and  $v(u_0,y,t)=v(u_0,y,0) + O(t)$ for small $t$ uniformly on compact sets. The function $y\mapsto v(u_0,y,0)$ is smooth and strictly positive.

Another importance of smooth points in our context is the fact that the sub-Riemannian distance $\dQuo(u_0,u)$ is smooth in a neighborhood of $(u_0,u)$ in $\{u_0\}\times Q(u_0)$ when $u$ is a smooth point \cite{agrachev_any_2009}.

\subsection{Small Time asymptotics on $M$}
We let $d=m-n$ denote the dimension of the fibers
of $\piQuo:Q(u_0)\rightarrow M$. For $y\in M$, the
distance $\dQuo \left(u_0,\piQuo^{-1}(y)\right)$ below is the minimal sub-Riemannian distance from $u_0$ to a point in the
fibre $\piQuo^{-1}(y)$ in $Q(u_0)$. For fixed $y$, we let $\Gamma\subset\piQuo^{-1}(y)$ be the set of minimizers of
  $\dQuo(u_0,\piQuo^{-1}(y))$, and we write
  \begin{equation*}
    z(u)=\dQuo(u_0,u)^2-\dQuo\left(u_0,\piQuo^{-1}(y)\right)^2
    \ .
  \end{equation*}
  Note that $z$ is 0 on $\Gamma$ and strictly positive on $\piQuo^{-1}(y)\setminus\Gamma$.
\begin{theorem}\label{thm:small}
Let $M $ be Riemannian, fix $u_0\in FM$ and let $y\in M$ be such that $\pi(u_0)\neq y$. Assume that $\Gamma=\{u_1,\dots,u_k\}$ consists of finitely many smooth points that are non-degenerate as critical points of $\dQuo(u_0,u)^2$, i.e.\ $\dQuo(u_0,u)^2$ has non-singular Hessian (in coordinates) at $u_i$. Then the transition density $p_t^M$ of $X_t$ satisfies
\begin{equation*}
    p_t^M(y)
    =
    (2\pi t)^{-n/2} e^{-\frac{\dQuo\left(u_0,\piQuo^{-1}(y)\right)^2}{2t}}
    w(y,t)
 \end{equation*}
  and $\lim_{t\rightarrow 0}{w}(y,t)$
  exists in $(0,\infty)$.
  \label{lem:smalltimeasymp}
\end{theorem}

\begin{proof}
  For each $u_i\in\Gamma$, let $U_{i}$ be a neighborhood of $u_i$ in $\piQuo^{-1}(y)$ such that the $U_i$ are disjoint and all points of $U_{i}$ are smooth. Moreover, let $K=\piQuo^{-1}(y)\setminus \Gamma_U$ where $\Gamma_U=\cup_i U_i$.
  Note that $K$ is compact.
	  From \eqref{eq:ptm} and \eqref{eq:subriemanniansmalltime}, we have
  \begin{align*}
    p_t^M(y)
    &=
    \int_{\Gamma_U}
    (2\pi t)^{-m/2}
    e^{-\frac{\dQuo(u_0,u)^2}{2t}}
    v(u_0,u,t) \mu_{\piQuo^{-1}(y)}(
    du)
    +
    \int_{K}
    p_t^{\Quo}(y) \mu_{\piQuo^{-1}(y)}(
    du)
    \\
    &=
    (2\pi t)^{-n/2}
    e^{-\frac{\dQuo\left(u_0,\piQuo^{-1}(y)\right)^2}{2t}}
    \left(
    w(y,t)
    +
    \tilde{w}(y,t)
    \right),
  \end{align*}
  where $ \mu_{\piQuo^{-1}(y)}$ is defined in Section \ref{transition} and
  \begin{equation}\label{wdef}
    w(y,t)
    =
    \int_{\Gamma_U}
    (2\pi t)^{-d/2}
    e^{-\frac{z(u)}{2t}}
    v(u_0,u,t)
     \mu_{\piQuo^{-1}(y)}(du)
  \end{equation}
 and
  \begin{equation}\label{wtildedef}
    \tilde{w}(y,t)
    =
    \int_{K}
    \frac{p_t^{\Quo}(u)}{(2\pi t)^{-n/2}e^{-\frac{\dQuo\left(u_0,\piQuo^{-1}(y)\right)^2}{2t}}}
     \mu_{\piQuo^{-1}(y)}(du)
    \ .
  \end{equation}
  To show that $\lim_{t\rightarrow 0}\tilde{w}(y,t)=0$, 
  we use that $\lim_{t\rightarrow0}2t\log(p_t^{\Quo}(u))=-\dQuo(u_0,u)^2$ uniformly
  on $K$ \cite{barilari_small-time_2012,leandre_minoration_1987}, and thus for any $\epsilon>0$,
  \begin{equation*}
    \frac{p_t^{\Quo}(u)}{(2\pi t)^{-n/2}e^{-\frac{\dQuo\left(u_0,\piQuo^{-1}(y)\right)^2}{2t}}}
    \le
    e^{-\frac{z(u)-\epsilon}{2t}}
  \end{equation*}
holds for $t$ sufficiently small.
  Since $z$ is bounded away from zero uniformly on $K$, the integrand is uniformly bounded by  $e^{-c/t}$ on $K$ for some $c>0$ when $t$ is sufficiently small.  Hence $\lim_{t\rightarrow 0}\tilde{w}(y,t)=0$.
   
Next we show that $w(y,t)$ has a limit in $(0,\infty)$ as $t\to \infty$.
Let $G^y$ be the Riemannian metric on $\piQuo^{-1}(y)$ that comes from regarding $\piQuo^{-1}(y)$ as a subgroup of $\GL(n)$.
 For each $i=1,\ldots,k$, let $\phi_i:V_i\rightarrow \piQuo^{-1}(y)$ be normal coordinate 
 charts centered at $u_i$ of the $\pi_{Q(u_0)}^{-1}(y)$-fibre with the metric $G^y$. We may assume $\phi_i(V_i)=U_i$.  Then
  \begin{align}\nonumber
    w(y,t)
    &=
    \sum_{i=1}^k
    \int_{V_i}
    (2\pi t)^{-d/2}
    e^{-\frac{z(\phi_i(\xi))}{2t}}
    v(u_0,\phi_i(\xi),t)
    |J_\xi\phi_i|
    d\xi
    \\ \label{wsum}
    &=
    (2\pi)^{-d/2}
    \sum_{i=1}^k
    \int_{t^{-\frac{1}{2}}V_i}
    e^{-\frac{z(\phi_i(\sqrt{t}\xi))}{2t}}
    v(u_0,\phi_i(\sqrt{t}\xi),t)
    |J_{\sqrt{t}\xi}\phi_i|
    d\xi
    \ ,
  \end{align}
  where $J\phi_i$ are the Jacobians of $\phi_i$, which are bounded on $V_i$ and satisfies $|J_0\phi_i|=1$.
  Define $c(u_0,u_i)=\inf_{u\in U_i\backslash\{u_i\}} \frac{z(u)}{d_{G^y}(u_i,u)^2}$. Because $u_i$ is non-degenerate, $c(u_0,u_i)>0$, and therefore
  \begin{align*}
    e^{-\frac{z(\phi_i(\sqrt{t}\xi))}{2t}}
    &\le
    e^{-\frac{c(u_0,u_i)d_{G^y}(u_i,\phi_i(\sqrt{t}\xi))^2}{2t}}
    \le
    e^{-\frac{c(u_0,u_i)\|\xi\|^2}{2}}
  \end{align*}
  for $\xi\in U_i$. The integrands in \eqref{wsum} are thus bounded by $C(u_0,y)e^{-c(u_0,y)\|\xi\|^2}$, which is integrable on $\RR^{d}$. 
  Furthermore, they are pointwise convergent since $z$ is smooth on $\Gamma_U$ and therefore
  $\lim_{t\to 0}\frac{z(\phi_i(\sqrt{t}\xi))}{t}=\frac{\partial^2}{\partial t^2}z(\phi_i(t\xi))_{| t=0} = \xi^T M_i \xi$, where $M_i$ is the Hessian matrix of $z\circ \phi_i$, which was assumed non-degenerate.
  By the dominated convergence theorem, we get
  \begin{align*}
    \lim_{t\to 0}w(y,t)
    &=
    (2\pi)^{-d/2}
    \sum_{i=1}^k
    \int_{\RR^{d}}
    e^{-\frac{1}{2}\xi^T M_i \xi}
    v(u_0,u_i,0)
    d\xi
    \ .
  \end{align*}
  Thus \begin{equation*}
    \lim_{t\to 0}w(y,t)
    =
    \sum_{i=1}^k
    v(u_0,u_i,0)
    w_i
  \end{equation*}
  with
  \begin{align*}
    w_i
    &=
    \int_{\RR^{d}}
       (2\pi)^{-d/2}
    e^{-\frac{1}{2}\xi^T M_i \xi}
    d\xi >0.  
  \end{align*}
	The claim follows because $v(u_0,u,0)>0$ everywhere.
\end{proof}
The theorem allows us to relate the sub-Riemannian distance $\dQuo$ to  $p_t^M$ and thereby relate $p_t^M$ to the most probable paths for the driving process.
In particular, the corollary below shows that the minimal $\dQuo$-distance to the $\piQuo^{-1}(y)$-fibre dominates in the transition density $p_t^M(y)$ for small $t$.
\begin{corollary}\label{cor:msmalltime}
  With the assumptions in Theorem \ref{thm:small}, the transition density of the target process satisfies
  \begin{equation*}
    \lim_{t\to 0}
    \frac{p_t^M(y)}{(2\pi t)^{-n/2}e^{-\frac{\dQuo\left(u_0,\piQuo^{-1}(y)\right)^2}{2t}}}
    =
   \lim_{t\to 0} w(y,t)
  \end{equation*}
  with $w(y,t)$ as in Theorem \ref{thm:small}, and
  \begin{equation*}
    \lim_{t\to 0}
    2t\log p_t^M(y)=-\dQuo\left(u_0,\piQuo^{-1}(y)\right)^2
    \ .
  \end{equation*}
\end{corollary}

\section{Applications in Statistics}\label{statistics}
In the classical Fr\'echet mean estimation, the mean is estimated by minimizing squared geodesic distances as explained in the introduction.
However, we can allow anisotropic covariance structures and  simultaneously encode mean and covariance in the distance measure by using the frame bundle and replacing the geodesic distance by the sub-Riemannian distance. In this way, the distances are weighted by the covariance in the sense that directions with high variance will contribute less to the distance than directions with small variance.

We can then define the  mean and covariance of a stochastic variable $X$ on $M$ to be the point $\sigma \in \Symp M$ that minimizes
\begin{equation*}
\mathbb{E} d_{\Symp M}(\sigma, q^{-1}(X))^2+F(\sigma),
\end{equation*}
where $F$ is a real valued map that prevents $\sigma$ from tending to $0$. Recall here that $\sigma$ encodes the precision matrix, and the covariance is found as the inverse $\sigma^{-1}$. As in the case of the Fr\'{e}chet mean, there is no guarantee that such a minimizer exists or is unique.
\begin{figure}[t]
  \begin{center}
    \includegraphics[width=.5\columnwidth,trim=80 280 80 250,clip]{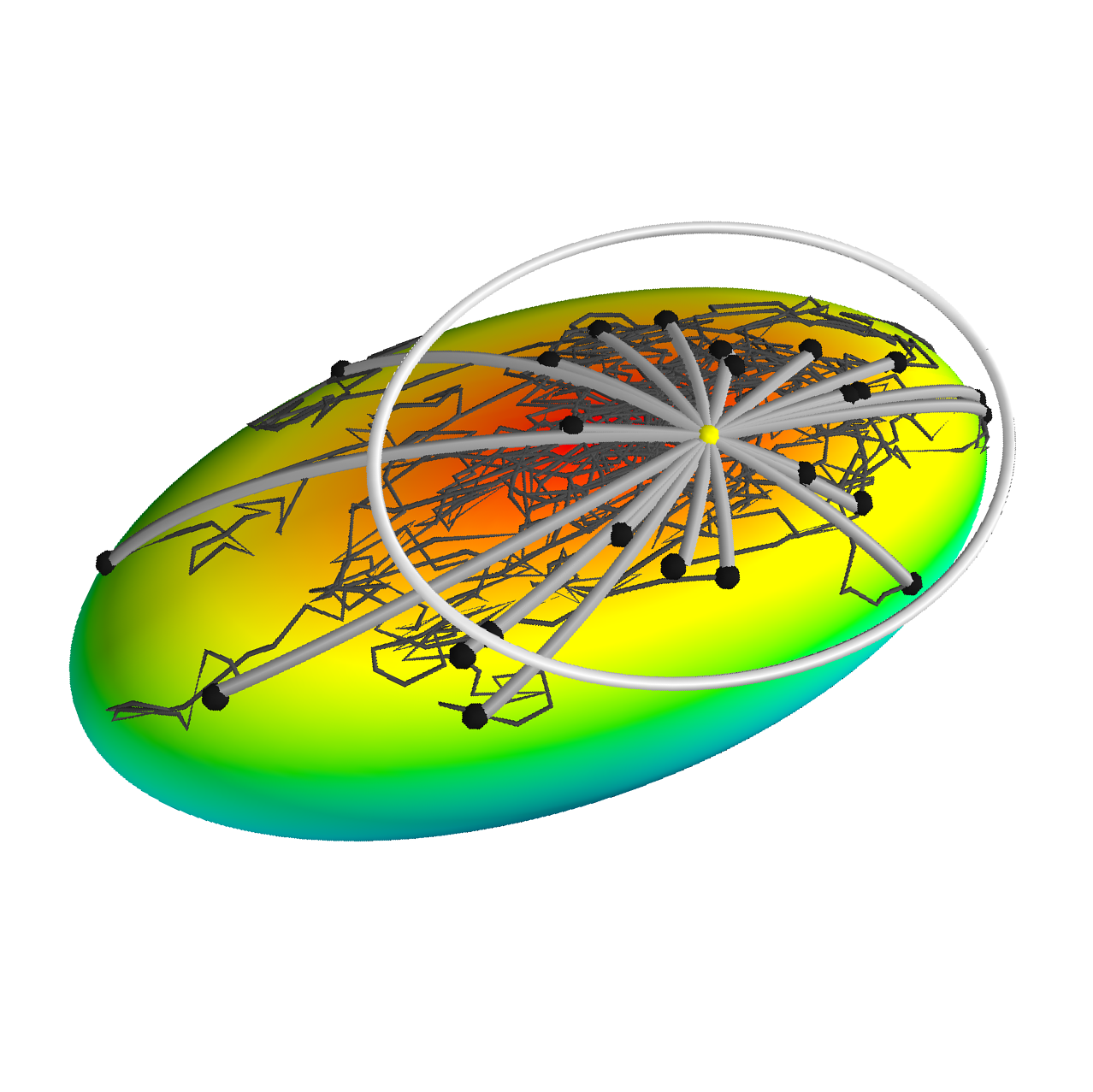}
  \end{center}
  \caption{
    Sampled data on an ellipsoid realized as endpoints of sample paths of the process $X_t$ (black lines and points). Mean $x_0=\pi(u_0)$ (green point) and covariance $\sigma^{-1}$ (ellipsis over mean) are estimated by minimizing \eqref{minimize}. The most probable paths for the driving process connects $x_0$ and the sample paths (gray lines) and minimize the distances $d_{\Symp M}\left(\sigma, q^{-1}(x_i)\right)$.
    }

  \label{fig:ellipsoidest}
\end{figure}

Given data points $x_1,\dots,x_N$ on $M$, we suggest to estimate the mean and covariance by
\begin{equation}\label{minimize}
  \textrm{argmin}_{\sigma \in \Symp M} \sum_{i=1}^Nd_{\Symp M}\left(\sigma, q^{-1}(x_i)\right)^2
  -\frac{N}{2}\log(\mathrm{det}_g\sigma)
\end{equation}
when $M$ is equipped with the Riemannian metric $g$. The term $-\frac{N}{2}\log(\mathrm{det}_g\sigma)$ is chosen such that \eqref{minimize} is exactly the maximum likelihood estimator of the mean and covariance in the Euclidean case, see e.g. \cite{tipping_probabilistic_1999}.
Figure~\ref{fig:ellipsoidest} shows estimated mean and covariance for sampled
data on an ellipsoid. The most probable paths for the driving process $W_t$ that
realize the distances $d_{\Symp M}\left(\sigma, q^{-1}(x_i)\right)$ are shown along
with sample paths from the process $X_t$ and samples drawn from $X_1$.

From a different viewpoint, the results of this paper allow a maximum-likelihood formulation that resembles the estimator \eqref{minimize}. Recall from Section \ref{sec:hormander} that when $M$ is Riemannian and $Q(u_0)$ satisfies the H\"{o}rmander condition, the process $X_t$ on $M$ has a density $p_t^M(\Sigma(u_0),x)$. 
We can then consider the estimation of $\Sigma(u_0)$ given the data points $x_1,\dots, x_N$ on $M$. The maximum likelihood approach suggests that we estimate our parameter $\sigma \in \Symp M$ by 
\begin{equation*}
\textrm{argmin}_{\sigma \in \Symp M} 
-\sum_{i=1}^N 
\log p_t^M(\sigma,x_i)
\ .
\end{equation*}
No explicit formula for the density $p_t^M$ is known in general, not even in the isotropic case, but heuristics or numerical approximations can be derived.
As an example, we showed in Corollary \ref{cor:msmalltime} that, in
non-degenerate situations, $\lim_{t\to 0}
2t\log p_t^M(y)=-\dQuo\left(u_0,\piQuo^{-1}(y)\right)^2$, and, from Theorem \ref{lem:smalltimeasymp}, the 
log-likelihood has the form
\begin{equation*}
-
\sum_{i=1}^N 
\left(
\frac{\dQuo\left(u_0,\piQuo^{-1}(x_i)\right)^2}{2t}
+\frac{n}{2}\log(2\pi t)-\log(w(x_i,t))
\right)
\end{equation*}
suggesting the estimator
\begin{equation*}
\textrm{argmin}_{u_0 \in FM}  
\sum_{i=1}^N 
\left(
\dQuo\left(u_0,\piQuo^{-1}(x_i)\right)^2
-2t\log(w(x_i,t))
\right)
\end{equation*}
with $t$ fixed. The terms $\log(w(x_i,t))$ are in general not a priori known.
They can be numerically estimated or heuristically set to $\log\det(u_0)_g/2$ which is
the normalization term of a Brownian motion with covariance $\Sigma(u_0)^{-1}$ in the
linear tangent space $T_{\pi(u_0)}M$. The latter approach results in the same estimator as
\eqref{minimize}. Note that both methods involve the length of the most probable
paths for the driving process. The error in the approximation of the likelihood
will in general be smaller for data that concentrates around $\pi(u_0)$.
The approach has been applied for mean/template and covariance estimation on concrete data in \cite{sommer_anisotropic_2015}.

The probability of a path as defined in Section \ref{sec4} yields yet another way of measuring the distance between points. Maximizing this probability corresponds to maximizing the Onsager-Machlup functional \eqref{energy}.
Alternatively, the limit of \eqref{OMeq} when $\epsilon \to 0$ can be viewed informally as a 'density' for the paths of the Brownian motion. The maximum-likelihood method would then suggest to maximize the Onsager-Machlup functional. Note that this density is only heuristic since it is concentrated on smooth paths while the Brownian motion is almost surely nowhere smooth. Moreover, it is not clear with respect to which measure it should be a density. 
In this setting, there are only results available when the Brownian motion is isotropic. In the isotropic case, the most probable paths are not quite geodesics. The Onsager-Machlup functional also tries to minimize the energy of paths but tends to favour paths through areas with high scalar curvature. It would be interesting to see how this affects the Fr\'{e}chet mean in practical computations. 
To obtain results for general Brownian motions, we could work with  the most probable paths of the driving process instead. This approach supports the estimator~\eqref{minimize}.

There are some challenges caused by the complex behaviour of the global structure of a sub-Riemannian geometry. 
Most geodesics with respect to $d_{FM}$ can be computed by the Hamiltonian equations, which makes practical computations possible. However, there may be geodesics that can not be computed this way. Moreover, the sub-Riemannian exponential map cannot be used to define normal coordiantes on $Q(u_0)$, which is often useful in ordinary Riemannian geometry to simplify computations. Finally, there is very little known about how the sub-Riemannian geometry changes when we vary the parameter $\sigma$.

\section*{Acknowledgments}
The authors wish to thank Peter W. Michor and Sarang Joshi for suggestions for the geometric
interpretation of the sub-Riemannian metric on $FM$ and discussions on diffusion
processes on manifolds. The work was supported by the Danish Council for Independent Research,
the CSGB Centre for Stochastic Geometry and Advanced Bioimaging funded by a grant from the Villum foundation,
and the Erwin Schr\"odinger Institute in Vienna.

\bibliographystyle{AIMS}
\bibliography{Zotero}

\providecommand{\href}[2]{#2}
\providecommand{\arxiv}[1]{\href{http://arxiv.org/abs/#1}{arXiv:#1}}
\providecommand{\url}[1]{\texttt{#1}}
\providecommand{\urlprefix}{URL }
\begin{thebibliography}{10}

\bibitem{agrachev_any_2009}
\newblock A.~A. Agrachev,
\newblock Any sub-{Riemannian} metric has points of smoothness,
\newblock \emph{Doklady Mathematics}, \textbf{79} (2009), 45--47.

\bibitem{barilari_small-time_2012}
\newblock D.~Barilari, U.~Boscain and R.~W. Neel,
\newblock Small-time heat kernel asymptotics at the sub-{Riemannian} cut locus,
\newblock \emph{Journal of Differential Geometry}, \textbf{92} (2012),
  373--416.

\bibitem{bryant_survey_1987}
\newblock R.~L. Bryant,
\newblock A survey of {Riemannian} metrics with special holonomy groups,
\newblock in \emph{Proceedings of the {International} {Congress} of
  {Mathematicians}}, vol. 1, 2,
\newblock Amer. Math. Soc., Berkeley, California, 1987,
\newblock 505--514.

\bibitem{emery_stochastic_1989}
\newblock M.~Emery,
\newblock \emph{Stochastic {Calculus} in {Manifolds}},
\newblock Universitext, Springer Berlin Heidelberg, Berlin, Heidelberg, 1989.

\bibitem{fletcher_statistics_2003}
\newblock P.~Fletcher, C.~Lu and S.~Joshi,
\newblock Statistics of shape via principal geodesic analysis on {Lie} groups,
\newblock in \emph{{CVPR} 2003}, vol.~1, 2003,
\newblock I--95--I--101 vol.1.

\bibitem{frechet_les_1948}
\newblock M.~Fréchet,
\newblock Les éléments aléatoires de nature quelconque dans un espace
  distancie,
\newblock \emph{Ann. Inst. H. Poincaré}, \textbf{10} (1948), 215--310.

\bibitem{fujita_onsager-machlup_1982}
\newblock T.~Fujita and S.-i. Kotani,
\newblock The {Onsager}-{Machlup} function for diffusion processes,
\newblock \emph{Journal of Mathematics of Kyoto University}, \textbf{22}
  (1982), 115--130.

\bibitem{grenander_computational_1998}
\newblock U.~Grenander and M.~I. Miller,
\newblock Computational anatomy: an emerging discipline,
\newblock \emph{Q. Appl. Math.}, \textbf{LVI} (1998), 617--694.

\bibitem{hsu_stochastic_2002}
\newblock E.~P. Hsu,
\newblock \emph{Stochastic {Analysis} on {Manifolds}},
\newblock American Mathematical Soc., 2002.

\bibitem{joyce_compact_2000}
\newblock D.~D. Joyce,
\newblock \emph{Compact {Manifolds} with {Special} {Holonomy}},
\newblock Oxford University Press, 2000.

\bibitem{leandre_minoration_1987}
\newblock R.~Léandre,
\newblock Minoration en temps petit de la densité d'une diffusion
  dégénérée,
\newblock \emph{Journal of Functional Analysis}, \textbf{74} (1987), 399--414.

\bibitem{michor_topics_2008}
\newblock P.~W. Michor,
\newblock \emph{Topics in {Differential} {Geometry}},
\newblock American Mathematical Soc., 2008.

\bibitem{modin_proceedings_2015}
\newblock K.~Modin and S.~Sommer,
\newblock \emph{Proceedings of {Math} {On} {The} {Rocks} {Shape} {Analysis}
  {Workshop} in {Grundsund}}, 2015,
\newblock \urlprefix\url{http://dx.doi.org/10.5281/zenodo.33558}.

\bibitem{mok_differential_1978}
\newblock K.-P. Mok,
\newblock On the differential geometry of frame bundles of {Riemannian}
  manifolds,
\newblock \emph{Journal Fur Die Reine Und Angewandte Mathematik}, \textbf{1978}
  (1978), 16--31.

\bibitem{montgomery_tour_2006}
\newblock R.~Montgomery,
\newblock \emph{A {Tour} of {Subriemannian} {Geometries}, {Their} {Geodesics}
  and {Applications}},
\newblock American Mathematical Soc., 2006.

\bibitem{pennec_intrinsic_2006}
\newblock X.~Pennec,
\newblock Intrinsic {Statistics} on {Riemannian} {Manifolds}: {Basic} {Tools}
  for {Geometric} {Measurements},
\newblock \emph{J. Math. Imaging Vis.}, \textbf{25} (2006), 127--154.

\bibitem{sommer_diffusion_2014}
\newblock S.~Sommer,
\newblock Diffusion {Processes} and {PCA} on {Manifolds},
\newblock Mathematisches Forschungsinstitut Oberwolfach, 2014,
\newblock
  \urlprefix\url{http://www.mfo.de/document/1440a/preliminary\_OWR\_2014\_44.pdf}.

\bibitem{sommer_anisotropic_2015}
\newblock S.~Sommer,
\newblock Anisotropic {Distributions} on {Manifolds}: {Template} {Estimation}
  and {Most} {Probable} {Paths},
\newblock in \emph{Information {Processing} in {Medical} {Imaging}},
\newblock Lecture {Notes} in {Computer} {Science}, Springer Berlin Heidelberg,
  2015.

\bibitem{sommer_evolution_2015}
\newblock S.~Sommer,
\newblock Evolution {Equations} with {Anisotropic} {Distributions} and
  {Diffusion} {PCA},
\newblock in \emph{Geometric {Science} of {Information}} (eds. F.~Nielsen and
  F.~Barbaresco),
\newblock no. 9389 in Lecture {Notes} in {Computer} {Science}, Springer
  International Publishing, 2015,
\newblock 3--11,
\newblock DOI: 10.1007/978-3-319-25040-3\_1.

\bibitem{strichartz_sub-riemannian_1986}
\newblock R.~S. Strichartz,
\newblock Sub-{Riemannian} geometry,
\newblock \emph{Journal of Differential Geometry}, \textbf{24} (1986),
  221--263.

\bibitem{sussmann_orbits_1973}
\newblock H.~J. Sussmann,
\newblock Orbits of {Families} of {Vector} {Fields} and {Integrability} of
  {Distributions},
\newblock \emph{Transactions of the American Mathematical Society},
  \textbf{180} (1973), 171--188.

\bibitem{tipping_probabilistic_1999}
\newblock M.~E. Tipping and C.~M. Bishop,
\newblock Probabilistic {Principal} {Component} {Analysis},
\newblock \emph{Journal of the Royal Statistical Society. Series B},
  \textbf{61} (1999), 611--622.

\bibitem{varadhan_behavior_1967}
\newblock S.~R.~S. Varadhan,
\newblock On the behavior of the fundamental solution of the heat equation with
  variable coefficients,
\newblock \emph{Communications on Pure and Applied Mathematics}, \textbf{20}
  (1967), 431--455.

\end{thebibliography}

\medskip
Received xxxx 20xx; revised xxxx 20xx.
\medskip

\end{document}